\documentclass[12pt,reqno]{article}
\usepackage{amssymb}
\usepackage[mathscr]{eucal}
\usepackage{amsmath,amssymb,latexsym,theorem,enumerate,ifthen}
\usepackage[english]{babel}
\usepackage{color,url}
\usepackage{appendix}
\usepackage{graphicx}
\usepackage{subcaption}

\newcommand{\NN}{\mathbb{N}}

\newcommand{\RR}{\mathbb{R}}

\newcommand{\bx}{{\boldsymbol{x}}}

\newcommand{\cE}{{\mathcal E}}

\newcommand{\cS}{{\mathcal S}}

\newcommand{\cT}{{\mathcal T}}

\newcommand{\cX}{{\mathcal X}}

\newcommand{\dd}{\mathrm{d}}
\newcommand{\ee}{\mathrm{e}}

\newcommand{\EE}{\operatorname{\mathbb{E}}}

\newcommand{\PP}{\operatorname{\mathbb{P}}}

\newcommand{\Diag}{\operatorname{Diag}}

\DeclareMathOperator*{\argmin}{arg\,min}

\newcommand{\comment}[1]{}

\renewcommand{\mid}{\,|\,}

\renewcommand{\leq}{\leqslant}
\renewcommand{\geq}{\geqslant}

\newcommand{\proofend}{\hfill\mbox{$\Box$}}

\numberwithin{equation}{section}

\theoremstyle{change} \theorembodyfont{\em}
\newtheorem{Lem}{Lemma.}[section]
\newtheorem{Thm}[Lem]{Theorem.}
\newtheorem{Pro}[Lem]{Proposition.}

\newtheorem{Cor}[Lem]{Corollary.}

\newtheorem{Def}[Lem]{Definition.}

\theorembodyfont{\rm}
\newtheorem{Ex}[Lem]{Example.}

\long\def\Eq#1#2{\ifthenelse{\equal{#1}{*}}
  {\begin{equation*}\begin{aligned}#2\end{aligned}\end{equation*}}
  {\begin{equation}\begin{aligned}\label{#1}#2\end{aligned}\end{equation}}}

\def\OnlyOnArXiv#1#2{\ifthenelse{\equal{#1}{Y}}{#2}{}}

\newenvironment{proof}{\noindent{\bf Proof.}}{\proofend}

\begin{document}

\begin{center}
 {\bfseries\Large Monotone representation and measurability \\[2mm] 
 of generalized $\psi$-estimators}

\vspace*{3mm}

{\sc\large
  M\'aty\'as $\text{Barczy}^{*,\diamond}$,
  Zsolt $\text{P\'ales}^{**}$ }

\end{center}

\vskip0.2cm

\noindent
 * HUN-REN–SZTE Analysis and Applications Research Group,
   Bolyai Institute, University of Szeged,
   Aradi v\'ertan\'uk tere 1, H--6720 Szeged, Hungary.

\noindent
 ** Institute of Mathematics, University of Debrecen,
    Pf.~400, H--4002 Debrecen, Hungary.

\noindent e-mail: barczy@math.u-szeged.hu (M. Barczy),
                  pales@science.unideb.hu  (Zs. P\'ales).

\noindent $\diamond$ Corresponding author.

\vskip0.2cm


{\renewcommand{\thefootnote}{}
\vspace*{0.2cm}

\vspace*{-10mm}

\begin{abstract}
We investigate the monotone representation and measurability of generalized $\psi$-estimators introduced by the authors in 2022.
Our first main result, applying the unique existence of a generalized $\psi$-estimator, allows us to construct this estimator in terms of a function $\psi$, which is decreasing in its second variable. We then interpret this result as a bridge from a nonconvex optimization problem to a convex one.
Further, supposing that the underlying measurable space (sample space) has a measurable diagonal and some additional assumptions on $\psi$, we show that the measurability of a generalized $\psi$-estimator is equivalent to the measurability of the corresponding function $\psi$ in its first variable.
\end{abstract}


\section{Introduction}
\label{section_intro}

In both theoretical and applied statistics, $M$-estimators play a fundamental role, and a special subclass, 
 the class of $\psi$-estimators (also called $Z$-estimators), is also in the heart of investigations.
The $M$-estimators (where the letter $M$ refers to ''maximum likelihood-type'') were introduced by Huber \cite{Hub64, Hub67}.
Let $(X,\cX)$ be a measurable space, $\Theta$ be a Borel subset of $\RR$, and $\varrho:X\times\Theta\to\RR$
 be a function such that for each $t\in\Theta$, the function $X\ni x\mapsto \varrho(x,t)$ is measurable
 with respect to the sigma-algebra $\cX$.
Let $(\xi_k)_{k\geq 1}$ be a sequence of i.i.d.\ random variables with values in $X$ such that the distribution of $\xi_1$ depends on an unknown parameter $\vartheta \in\Theta$.
For each $n\geq 1$, Huber \cite{Hub64, Hub67} introduced an important estimator of $\vartheta$ based on the observations
$\xi_1,\ldots,\xi_n$ as a solution of the following minimization problem:
 \begin{align}\label{help_M_est_min_problem}
   \inf_{t\in\Theta}\sum_{i=1}^n \varrho(\xi_i,t),
 \end{align}
provided that such a solution exists.
One calls such a solution an $M$-estimator of the unknown parameter $\vartheta\in\Theta$ based on
 the i.i.d.\ observations $\xi_1,\ldots,\xi_n$.
For historical fidelity, we note that Huber \cite{Hub64} considered the special case when $X:=\RR$, $\Theta:=\RR$, and the function $\varrho$ depends only on $x-t$,
 i.e., $\varrho(x,t):=f(x-t)$, $x\in\RR$, $t\in\Theta$,
 with some given nonconstant function $f:\RR\to\RR$.
Turning back to the general case, under suitable regularity assumptions, the minimization problem  \eqref{help_M_est_min_problem} can be solved 
 by setting the derivative of the objective function (with respect to the unknown parameter) equal to zero:
 \[
 \sum_{i=1}^n \partial_2\varrho(\xi_i,t)=0, \qquad t\in\Theta,
 \]
 where $\partial_2 \varrho$ denotes the (partial) derivative of $\varrho$ with respect to its second variable.
In the statistical literature, $\partial_2\varrho$ and a solution of the equation above 
 are often denoted by $\psi$ and $\widehat\vartheta_{n,\psi}(\xi_1,\ldots,\xi_n)$, respectively,
 and hence in this case an $M$-estimator is often called a $\psi$-estimator, while other authors call it a $Z$-estimator (the letter $Z$ refers to ''zero'').
For a detailed exposition of $M$-estimators and $\psi$-estimators ($Z$-estimators), see, e.g., Kosorok \cite[Sections 2.2.5 and 13]{Kos} or van der Vaart \cite[Section 5]{Vaa}.

In our recent paper Barczy and P\'ales \cite{BarPal2}, we introduced the notion of weighted generalized $\psi$-estimators
 (recalled below in Definition \ref{Def_sign_change}), and we studied their existence and uniqueness,
 see also Theorem \ref{Thm_M_est_uniq}, Corollary \ref{Cor1} and Proposition \ref{Pro_M_est_uniq}.
In Barczy and P\'ales \cite{BarPal4}, we established several properties of (weighted) generalized $\psi$-estimators: mean-type, monotonicity and sensitivity properties, bisymmetry-type inequality and some asymptotic and continuity properties as well.
This paper is devoted to derive two more essential theoretical properties of such estimators related to the decreasingness of $\psi$ in its second variable and the measurability of $\psi$ in its first variable, 
 see Theorems \ref{Thm_conjecture} and \ref{Thm_measurability}, respectively.
 We then interpret Theorem \ref{Thm_conjecture} in the frame of optimization theory, see Theorem \ref{Thm_convexity}.

Throughout this paper, let $\NN$, $\RR$, $\RR_+$ and $\RR_{++}$ denote the sets of positive integers, real numbers, non-negative real numbers and positive real numbers, respectively.
An interval $\Theta\subseteq\RR$ will be called nondegenerate if it contains at least two distinct points.
Given a nondegenerate interval $\Theta\subseteq \RR$ and a function $f:\Theta\to \RR$, let us define
 \[
   \argmin_{t\in\Theta} f(t) := \big\{ t\in\Theta : f(s)\geq f(t)\;\; \text{for all $s\in\Theta$} \big\}.
 \]
Given a function $f:\RR^2\to\RR$, the partial derivatives of $f$ with respect to its first and second variables are denoted by
 $\partial_1 f$ and $\partial_2 f$, respectively.
For each $n\in\NN$, let us also introduce the set $\Lambda_n:=\RR_+^n\setminus\{(0,\ldots,0)\}$.

\begin{Def}\label{Def_sign_change}
Let $\Theta$ be a nondegenerate open interval of $\RR$. For a function $f:\Theta\to\RR$, consider the following three level sets
\[
  \Theta_{f>0}:=\{t\in \Theta: f(t)>0\},\qquad
  \Theta_{f=0}:=\{t\in \Theta: f(t)=0\},\qquad
  \Theta_{f<0}:=\{t\in \Theta: f(t)<0\}.
\]
We say that $\vartheta\in\Theta$ is a \emph{point of sign change (of decreasing type) for $f$} if
 \[
 f(t) > 0 \quad \text{for $t<\vartheta$,}
   \qquad \text{and} \qquad
    f(t)< 0 \quad  \text{for $t>\vartheta$.}
 \]
\end{Def}

Note that there can exist at most one element $\vartheta\in\Theta$ which is a point of sign change for $f$.
Further, if $f$ is continuous at a point $\vartheta$ of sign change, then $\vartheta$ is the unique zero of $f$.

Let $X$ be a nonempty set, $\Theta$ be a nondegenerate open interval of $\RR$.
Let $\Psi(X,\Theta)$ denote the class of real-valued functions $\psi:X\times\Theta\to\RR$ such that, for all $x\in X$, there exist $t_+,t_-\in\Theta$ such that $t_+<t_-$ and $\psi(x,t_+)>0 >\psi(x,t_-)$.
Roughly speaking, a function $\psi\in\Psi(X,\Theta)$ satisfies the following property: for all $x\in X$, the function $t\ni\Theta\mapsto \psi(x,t)$ changes sign (from positive to negative) on the interval $\Theta$ at least once.

\begin{Def}\label{Def_Tn}
We say that a function $\psi\in\Psi(X,\Theta)$
  \begin{enumerate}[(i)]
    \item \emph{possesses the property $[C]$ (briefly, $\psi$ is a $C$-function) if
           it is continuous in its second variable, i.e., if, for all $x\in X$,
           the mapping $\Theta\ni t\mapsto \psi(x,t)$ is continuous.}
    \item \emph{possesses the property $[T_n]$ (briefly, $\psi$ is a $T_n$-function)
           for some $n\in\NN$} if there exists a mapping $\vartheta_{n,\psi}:X^n\to\Theta$ such that,
           for all $\pmb{x}=(x_1,\dots,x_n)\in X^n$ and $t\in\Theta$,
           \begin{align*}
             \psi_{\pmb{x}}(t):=\sum_{i=1}^n \psi(x_i,t) \begin{cases}
                 > 0 & \text{if $t<\vartheta_{n,\psi}(\pmb{x})$,}\\
                 < 0 & \text{if $t>\vartheta_{n,\psi}(\pmb{x})$},
            \end{cases}
           \end{align*}
          that is, for all $\pmb{x}\in X^n$, the value $\vartheta_{n,\psi}(\pmb{x})$ is a point of sign change for the function $\psi_{\pmb{x}}$. If there is no confusion, instead of $\vartheta_{n,\psi}$ we simply write $\vartheta_n$.
          We may call $\vartheta_{n,\psi}(\pmb{x})$ as a generalized $\psi$-estimator for
         some unknown parameter in $\Theta$ based on the realization $\bx=(x_1,\ldots,x_n)\in X^n$. If, for each $n\in\NN$, $\psi$ is a $T_n$-function, then we say that \emph{$\psi$ possesses the property $[T]$ (briefly, $\psi$ is a $T$-function)}.
    \item \emph{possesses the property $[Z_n]$ (briefly, $\psi$ is a $Z_n$-function) for some $n\in\NN$} if it is a $T_n$-function and
    \[
   \psi_{\pmb{x}}(\vartheta_{n,\psi}(\pmb{x}))=\sum_{i=1}^n \psi(x_i,\vartheta_{n,\psi}(\pmb{x}))= 0
    \qquad \text{for all}\quad \pmb{x}=(x_1,\ldots,x_n)\in X^n.
    \]
    If, for each $n\in\NN$, $\psi$ is a $Z_n$-function, then we say that \emph{$\psi$ possesses the property $[Z]$ (briefly, $\psi$ is a $Z$-function)}.
    \item \emph{possesses the property $[T_n^{\pmb{\lambda}}]$ for some $n\in\NN$ and $\pmb{\lambda}=(\lambda_1,\ldots,\lambda_n)\in\Lambda_n$ (briefly, $\psi$ is a $T_n^{\pmb{\lambda}}$-function)} if there exists a mapping $\vartheta_{n,\psi}^{\pmb{\lambda}}:X^n\to\Theta$ such that, for all $\pmb{x}=(x_1,\dots,x_n)\in X^n$ and $t\in\Theta$,
          \begin{align*}
           \psi_{\pmb{x},\pmb{\lambda}}(t):= \sum_{i=1}^n \lambda_i\psi(x_i,t) \begin{cases}
                 > 0 & \text{if $t<\vartheta_{n,\psi}^{\pmb{\lambda}}(\pmb{x})$,}\\
                 < 0 & \text{if $t>\vartheta_{n,\psi}^{\pmb{\lambda}}(\pmb{x})$},
             \end{cases}
           \end{align*}
           that is, for all $\pmb{x}\in X^n$, the value $\vartheta_{n,\psi}^{\pmb{\lambda}}(\pmb{x})$ is
           a point of sign change for the function $\psi_{\pmb{x},\pmb{\lambda}}$.
           If there is no confusion, instead of $\vartheta_{n,\psi}^{\pmb{\lambda}}$ we simply write $\vartheta_n^{\pmb{\lambda}}$.
          We may call $\vartheta_{n,\psi}^{\pmb{\lambda}}(\pmb{x})$
          as a weighted generalized $\psi$-estimator for some unknown parameter in $\Theta$ based
          on the realization $\bx=(x_1,\ldots,x_n)\in X^n$ and weights $(\lambda_1,\ldots,\lambda_n)\in\Lambda_n$.
   \end{enumerate}
\end{Def}

Similarly to the property $[C]$ introduced in Definition \ref{Def_Tn}, we say that a function $\psi\in\Psi(X,\Theta)$ is
 increasing, continuously differentiable and convex in its second variable, if, for all $x\in X$, the mapping $\Theta\ni t\mapsto \psi(x,t)$ is increasing, continuously differentiable and convex, respectively.
It can be seen that if $\psi$ is continuous in its second variable, and, for some $n\in\NN$, it is a $T_n$-function, then it also a $Z_n$-function.
Further, if $\psi\in\Psi(X,\Theta)$ is a $T_n$-function for some $n\in\NN$, then $\vartheta_{n,\psi}$ is symmetric in the sense that
$\vartheta_{n,\psi}(x_1,\ldots,x_n) = \vartheta_{n,\psi}(x_{\pi(1)},\ldots,x_{\pi(n)})$ holds for all $x_1,\ldots,x_n\in X$ and all permutations $(\pi(1),\ldots,\pi(n))$ of $(1,\ldots,n)$.

Given properties $[P_1], \ldots, [P_q]$ introduced in Definition~\ref{Def_Tn} (where $q\in\NN$), the subclass of $\Psi(X,\Theta)$ consisting of
 elements possessing the properties $[P_1],\ldots,[P_q]$ will be denoted by $\Psi[P_1,\ldots,P_q](X,\Theta)$.

In Section 4 of Barczy and P\'ales \cite{BarPal2} and in Section 3 of Barczy and P\'ales \cite{BarPal3},
 one can find several examples for (generalized) $\psi$-estimators that are used
 in statistics, in particular, for estimating parameters of notable distributions.
To present some new examples compared to the ones in \cite{BarPal3, BarPal2}, next, we discuss two particular (usual) $\psi$-estimators, the so-called $t$-score moment estimator
 and the trimmed moment estimator, that are popular in robust statistics.
In particular, these two estimators can be well-applied for estimating parameters of heavy-tailed distributions. 

\begin{Ex}
First, we recall the notion of $t$-score moment estimators based on Stehlík et al.\ \cite[Section 2.1]{StePotWalFab}.
Let $a,b\in\RR$ be such that $a<b$, and $\xi$ be an absolutely continuous random variable such that 
 $\PP(\xi\in(a,b))=1$.
Assume that the distribution of $\xi$ depends on an unknown parameter $\vartheta\in\Theta$, where $\Theta$ is a Borel subset of $\RR$.
Let $f_\vartheta:\RR\to\RR$ denote the density function of $\xi$.
Assume that $f_\vartheta$ is positive and continuously differentiable on $(a,b)$.
Further, let $h:(a,b)\to\RR$ be a continuously differentiable function such that $h(x)>0$, $x\in(a,b)$.
The $t$-score (transformation-based score) function of $f_\vartheta$ is defined by $\cT_\vartheta:(a,b)\to\RR$,
 \[
  \cT_\vartheta(x):=\frac{1}{f_\vartheta(x)} \left(-\frac{f_\vartheta}{h}\right)'(x),\qquad x\in(a,b).
 \] 
Given an $n\in\NN$ and independent and identically distributed random variables $\xi_1,\ldots,\xi_n$ such that $\xi_1$ has the same distribution as $\xi$, the parameter $\vartheta$ can be estimated as a solution of the equation
 \begin{align*}
 \frac1n \sum_{i=1}^n \cT_\vartheta(\xi_i) = \EE(\cT_\vartheta(\xi_1)),\qquad \vartheta\in\Theta,
 \end{align*}
 provided that such a solution exists. 
Stehlík et al.\ \cite[equations (7) and (8)]{StePotWalFab} call this estimator as the $t$-score moment estimator of $\vartheta$.
 Note that the $t$-score moment estimator is a particular (usual) $\psi$-estimator corresponding to the function 
 $\psi:(a,b)\times \Theta\to\RR$ defined by $\psi(x,\vartheta):=\cT_\vartheta(x)-\EE(\cT_\vartheta(\xi_1))$, $x\in(a,b)$, $\vartheta\in\Theta$.
 
For the notion of trimmed moment estimators, see Brazauskas et al.\ \cite[Section 2.1]{BraJonZit}.    
Here we only note that this estimator is also a solution of an appropriate equation in terms of $\vartheta\in\Theta$,
 so it is also a particular (usual) $\psi$-estimator, see Brazauskas et al.\ \cite[Note 2.1]{BraJonZit}.
\proofend
\end{Ex}

Next, we recall some necessary and sufficient conditions for the properties $[T_n]$ and $[T_n^{\pmb{\lambda}}]$ introduced in Definition \ref{Def_Tn}, which were proved in Barczy and P\'ales \cite{BarPal2}.
For all $x,y\in X$ with $\vartheta_1(x)<\vartheta_1(y)$, let us introduce the map
 \begin{align}\label{function_newhanyados0}
   (\vartheta_1(x),\vartheta_1(y))\ni t \mapsto  -\frac{\psi(x,t)}{\psi(y,t)}.
 \end{align}

The following result can be found in Barczy and P\'ales \cite[ parts (iv), (v) and (vi) of Theorem 1]{BarPal2}.

\begin{Thm}\label{Thm_M_est_uniq}
Let $X$ be a nonempty set, $\Theta$ be a nondegenerate open interval of $\RR$, and $ \psi\in\Psi[T_1](X,\Theta)$.
\begin{enumerate}[(i)]\itemsep=-2pt
 \item If $\psi\in\Psi[T_n](X,\Theta)$ for infinitely many $n\in\NN$, then for each $x,y\in X$ with $\vartheta_1(x)<\vartheta_1(y)$, the function \eqref{function_newhanyados0} is increasing.
 \item If $\psi\in\Psi[T_2^{\pmb{\lambda}}](X,\Theta)$ for each $\pmb{\lambda}\in\Lambda_2$, then for each $x,y\in X$ with $\vartheta_1(x)<\vartheta_1(y)$, the function \eqref{function_newhanyados0} is strictly increasing.
\item If $\psi\in\Psi[Z_1](X,\Theta)$ and, for each $x,y\in X$ with $\vartheta_1(x)<\vartheta_1(y)$, the function \eqref{function_newhanyados0} is strictly increasing, then $ \psi\in\Psi [T_n^{\pmb{\lambda}}](X,\Theta)$ for each $n\in\NN$ and $\pmb{\lambda}\in\Lambda_n$.
\end{enumerate}
\end{Thm}

We emphasize that, for $\psi\in\Psi[T_1](X,\Theta)$, the assertion (iii) of Theorem \ref{Thm_M_est_uniq}
 provides a sufficient condition for the existence and uniqueness of a weighted generalized $\psi$-estimator.
In the following statement, under the property $[Z_1]$, we recall two equivalent conditions for 
 the existence and uniqueness of a weighted generalized $\psi$-estimator due to Barczy and P\'ales \cite[Corollary 1]{BarPal2}.

\begin{Cor}\label{Cor1}
Let $X$ be a nonempty set, $\Theta$ be a nondegenerate open interval of $\RR$, and $\psi\in\Psi[Z_1](X,\Theta)$.
Then the following assertions are equivalent:
\begin{enumerate}[(i)]\itemsep=-2pt
\item For each $x,y\in X$ with $\vartheta_1(x)<\vartheta_1(y)$, the function \eqref{function_newhanyados0} is strictly increasing.

\item For each $\pmb{\lambda}\in\Lambda_2$, we have $\psi\in\Psi[T_2^{\pmb{\lambda}}](X,\Theta)$.
\item For each $n\in\NN$ and $\pmb{\lambda}\in\Lambda_n$, we have $\psi\in\Psi[T_n^{\pmb{\lambda}}](X,\Theta)$.
\end{enumerate}
\end{Cor}

In part (ii) of the next proposition, we provide a sufficient condition (which does not involve the property $[Z_1]$)
 under which $\psi$ has the property $[T_n^{\pmb{\lambda}}]$ for each $n\in\NN$ and $\pmb{\lambda}\in\Lambda_n$.
This result can be found in Barczy and P\'ales \cite[Proposition 2]{BarPal2}.

\begin{Pro}\label{Pro_M_est_uniq}
Let $X$ be a nonempty set, $\Theta$ be a nondegenerate open interval of $\RR$, and
$\psi\in\Psi[T_1](X,\Theta)$.
\begin{enumerate}[(i)]\itemsep=-2pt
 \item If for each $x\in X$, the function $\Theta\ni t\mapsto \psi(x,t)$ is (strictly) decreasing, then for each $x,y\in X$ with $\vartheta_1(x)<\vartheta_1(y)$, the
       function \eqref{function_newhanyados0}  is (strictly) increasing.
 \item If for each $x\in X$, the function $\Theta\ni t\mapsto \psi(x,t)$ is strictly decreasing, then $\psi\in\Psi[T_n^{\pmb{\lambda}}](X,\Theta)$ for each $n\in\NN$ and $\pmb{\lambda}\in\Lambda_n$.
\end{enumerate}
\end{Pro}

The remaining part of the paper is divided into two sections.
In Section \ref{Sec_decreasing}, we formulate and prove a converse of the statement
 in part (i) of Proposition 2 in Barczy and P\'ales \cite{BarPal2} (see also part (i) of Proposition \ref{Pro_M_est_uniq}).
Roughly speaking, given a function $\psi\in\Psi[Z_1](X,\Theta)$, under some regularity assumptions on $\psi$ in its second variable
 and on the map $\vartheta_1$, we prove that the property that the function \eqref{function_newhanyados0} is increasing
 for all $x,y\in X$ with $\vartheta_1(x)<\vartheta_1(y)$ is equivalent to the existence of a positive and locally absolutely continuous function $p:\Theta\to\RR$ such that, for all $z\in X$, the map $\Theta\ni t\mapsto p(t)\psi(z,t)$ is decreasing on $\Theta$ and locally absolutely continuous
 on $\Theta\setminus\{\vartheta_1(z)\}$, see Theorem \ref{Thm_conjecture}.
We also demonstrate Theorem \ref{Thm_conjecture} by investigating the maximum likelihood estimator
 of the mean of a normally distributed random variable with a given variance, and that of the shape parameter of a random variable with Pareto distribution having scale parameter $1$, see Examples \ref{Ex_mle} and \ref{Ex_pareto}, respectively.
Furthermore, we present an application of Theorem \ref{Thm_conjecture} related to convex optimization theory.
Namely, supposing that we are given a family of minimization problems related to continuous functions that are strictly decreasing up to a point and strictly increasing from that point on (such functions belong to the set of quasi-convex functions),
 we show that there exists a family of minimization problems related to some appropriate convex functions such that
 the unique solutions of the corresponding members of the two family of minimization problems coincide, see Theorem \ref{Thm_convexity}.
In Section \ref{Sec_measurability}, supposing that the underlying measurable space $X$ has a measurable diagonal,
 $\psi$ has the property $[Z]$ and another additional assumption on $\psi$, we prove that the measurability of
 a generalized $\psi$-estimator is equivalent to the measurability of the corresponding function $\psi$ in its first variable,
 see Theorem \ref{Thm_measurability}.
For the proof of Theorem \ref{Thm_measurability}, we needed to show the following property of measurable spaces,
 which may be interesting on its own right as well.
Given a measurable space $(X,\cX)$ with a measurable diagonal (i.e., $\{(x,x) : x\in X \}$ belongs to the product sigma-field $\cX\times\cX=\cX^2$),
 we show that, for each $n\in\NN$, the set $\big\{ (x_1,\ldots,x_n)\in X^n : x_1=\cdots = x_n \big\}$
 belongs to the $n$-fold product sigma-algebra $\cX^n$, see Lemma \ref{Lem_nfold_diag_measurable}.

\section{Generalized $\psi$-estimators and the decreasingness of $\psi$ in its second variable}
\label{Sec_decreasing}

In this section, we investigate how we could formulate and prove a converse of the statement in part (i) of Proposition 2 in Barczy and P\'ales \cite{BarPal2} (see also Proposition \ref{Pro_M_est_uniq}). Before formulating our result, we recall some notions from real analysis.
Given an open set $U\subseteq\RR$, we say that a function $h:U\to\RR$ is locally absolutely continuous if its restriction to any compact subinterval of $U$ is absolutely continuous.
If $h:U\to\RR$ is locally absolutely continuous, then it is continuous as well.
On the other hand, if $h:U\to\RR$ is continuously differentiable, then it is locally Lipschitz continuous and locally absolutely continuous as well.
Recall that if $h:U\to\RR$ is locally absolutely continuous, then $h'$ exists almost everywhere on $U$ and, for all $x,\tau\in U$ with $[x,\tau]\subseteq U$ or $[\tau,x]\subseteq U$, we have $h(x) = h(\tau) + \int_\tau^x h'(t)\,\dd t$, which we will refer to as the Newton--Leibniz formula (see, e.g., Hewitt and Stromberg \cite[Theorems (18.16) and (18.17)]{HewStr}.
We note that, in general, the almost everywhere differentiability of $h$ does not imply the validity of the Newton--Leibniz formula (for example, one can choose $h$ as the Cantor staircase function on $[0,1]$, which has zero derivative almost everywhere, and takes the values $0$ and $1$ at $0$ and $1$, respectively).

Recall also that, given a nondegenerate open interval $\Theta$ of $\RR$, an extended real-valued function $h:\Theta\to\overline\RR$
 is lower semicontinuous at a point $t_0\in\Theta$ if, for all $y\in\RR$ with $y<h(t_0)$,
 there exists a neighbourhood $U\subseteq \Theta$ of $t_0$ such that $y<h(t)$ holds for all $t\in U$ (equivalently, $\liminf_{t\to t_0} h(t)\geq h(t_0)$).
Similarly, an extended real-valued function $h:\Theta\to\overline\RR$ is upper semicontinuous at a point $t_0\in\Theta$ if for all $y\in\RR$ with $y>h(t_0)$,
 there exists a neighbourhood $U\subseteq \Theta$ of $t_0$ such that $y>h(t)$ for all $t\in U$
 (equivalently, $\limsup_{t\to t_0} h(t)\leq h(t_0)$).
A function $h:\Theta\to\overline\RR$ is called upper (lower) semicontinuous if it is upper (lower) semicontinuous at every point of $\Theta$.
From the definitions, it readily follows that an extended real-valued lower (upper) semicontinuous function on $\Theta$ is locally bounded from below (above) on $\Theta$, i.e., for every point of $\Theta$ there exists a neighbourhood on which the function in question is bounded below (above).
Furthermore, it is also known that an extended real-valued lower (upper) semicontinuous function on $\Theta$ is Borel measurable.
One can easily check that a function $h:\Theta\to\RR$ is continuous if and only if it is lower and upper semicontinuous.

Our forthcoming result is motivated by Dar\'oczy and P\'ales \cite{DarPal82} and P\'ales \cite{Pal88a}.

\begin{Thm}\label{Thm_conjecture}
Let $X$ be a nonempty set, $\Theta$ be a nondegenerate open interval of $\RR$, and $\psi\in\Psi[Z_1](X,\Theta)$.
Assume that, for all $z\in X$, the map $ \Theta\setminus\{\vartheta_1(z)\}\ni t\mapsto \psi(z,t)$ is continuously differentiable and,
 for all $t\in\Theta$, there exist $x,y\in X$ such that $\vartheta_1(x)<t< \vartheta_1(y)$.
Then the following two assertions are equivalent:
\vspace{-3mm}
\begin{enumerate}[(i)]
 \item For all $x,y\in X$ with $\vartheta_1(x)<\vartheta_1(y)$, the map \eqref{function_newhanyados0} is increasing.
 \item There exists a positive and locally absolutely continuous function $p:\Theta\to\RR$ such that,
       for all $z\in X$, the map
       \begin{equation}\label{prod_map}
          \Theta\ni t\mapsto p(t)\psi(z,t)
       \end{equation}
       is decreasing on $\Theta$ and locally absolutely continuous on $\Theta\setminus\{\vartheta_1(z)\}$.
\end{enumerate}
\end{Thm}

\begin{proof}
First, assume that assertion (i) holds.
Then, for all $x,y\in X$ with $\vartheta_1(x)<\vartheta_1(y)$, the increasingness and the differentiability of the map \eqref{function_newhanyados0} imply that,
 for all $t\in (\vartheta_1(x),\vartheta_1(y))$, we have
 \[
 \frac{\dd}{\dd t} \left( -\frac{\psi(x,t)}{\psi(y,t)} \right)
       = -\frac{\partial_2\psi(x,t)\cdot \psi(y,t) - \psi(x,t)\cdot \partial_2\psi(y,t)}{(\psi(y,t))^2} \geq 0,
 \]
 which yields that
 $$
   \partial_2 \psi(x,t)\cdot\psi(y,t)
   \leq \partial_2 \psi(y,t)\cdot\psi(x,t).
 $$
Therefore, for all $t\in\Theta$ and $x,y\in X$ with $\vartheta_1(x)<t<\vartheta_1(y)$, we have $\psi(x,t)<0$ and $\psi(y,t)>0$, and hence
\begin{equation}\label{help_3}
   \frac{\partial_2 \psi(y,t)}{\psi(y,t)}\leq\frac{\partial_2 \psi(x,t)}{\psi(x,t)}.
\end{equation}
Define $q_*:\Theta\to\overline{\RR}$ and $q^*:\Theta\to\overline{\RR}$ by
 \begin{align}\label{formulae_qs}
  \begin{split}
   &q_*(t):=\sup\bigg\{\frac{\partial_2 \psi(y,t)}{\psi(y,t)}\,\bigg|\,y\in X,\,t<\vartheta_1(y)\bigg\},\\[1mm]
   &q^*(t):=\inf\bigg\{\frac{\partial_2 \psi(x,t)}{\psi(x,t)}\,\bigg|\,x\in X,\,\vartheta_1(x)<t\bigg\}
  \end{split}
 \end{align}
for $t\in\Theta$. In view of the assumptions of the theorem, we can see that $q_*(t)>-\infty$ and $q^*(t)<\infty$ for all $t\in \Theta$.
Since the (pointwise) supremum of an arbitrary family of lower semicontinuous functions is also lower semicontinuous, the (pointwise) infimum of an arbitrary family of upper semicontinuous functions is also upper semicontinuous, and, by the assumptions, for all $z\in X$, the map $\Theta\setminus\{\vartheta_1(z)\}\ni t \mapsto \frac{\partial_2 \psi(z,t)}{\psi(z,t)}$ is continuous (and hence lower and upper semicontinuous), we can see that $q_*$ is a lower semicontinuous function on $\Theta$ and $q^*$ is upper semicontinuous on $\Theta$.
Consequently, $q_*$ and $q^*$ are also measurable.
It follows from these properties that $q_*$ is locally bounded from below on $\Theta$, and $q^*$ is locally bounded from above on $\Theta$.
 Furthermore, by the definitions of these functions and due to the inequality \eqref{help_3}, for all $t\in\Theta$ and $x,y\in X$ with $\vartheta_1(x)<t<\vartheta_1(y)$, we get
 \begin{align}\label{help_4}
   \frac{\partial_2 \psi(y,t)}{\psi(y,t)}
   \leq q_*(t)\leq q^*(t)
   \leq \frac{\partial_2 \psi(x,t)}{\psi(x,t)}.
 \end{align}
Therefore, using the assumption that, for all $t\in\Theta$, there exist $x,y\in X$ such that $\vartheta_1(x)<t<\vartheta_1(y)$, we have both functions $q_*$ and $q^*$ are finite valued, moreover locally bounded on $\Theta$.
Consequently, $q_*$ is bounded on every compact subinterval of $\Theta$, and, using that a bounded and Borel measurable function defined on a Borel measurable set of $\RR$ having finite Lebesgue measure is Lebesgue integrable, we get that $q_*$ is Lebesgue integrable on every compact subinterval of $\Theta$.
Therefore, by the fundamental theorem of Lebesgue integral calculus, for all $\tau\in\Theta$,
 the map $\Theta\ni t\mapsto \int_\tau^t q_*(s)\,\dd s\in\RR$ is locally absolutely continuous.

Fix $\tau\in\Theta$ and define $p:\Theta\to\RR_{++}$ by
 \begin{align}\label{formula_p}
   p(t):=\exp\bigg(-\int_\tau^t q_*(s)\,\dd s\bigg), \qquad t\in\Theta.
 \end{align}
Then $p$ is locally absolutely continuous (in particular, it is continuous), because it is known that the composition of a locally Lipschitz continuous function
 (in our case the exponential function) and a locally absolutely continuous function is locally absolutely continuous.
Consequently, $p$ is differentiable almost everywhere on $\Theta$, and $p'(t) = - p(t)q_*(t)$ for a.e.\ $t\in\Theta$.
Using \eqref{help_4}, we have that
 \begin{align}\label{help_5}
   -\frac{p'(t)}{p(t)} = q_*(t) \leq q^*(t),
   \qquad \hbox{a.e. }t\in\Theta.
 \end{align}
Let $\Theta^\psi$ denote the set of those points $t\in\Theta$, where the above equality and inequality hold.
Then we have that the Lebesgue measure of $\Theta\setminus\Theta^\psi$ is zero.
By the definitions of $q_*$ and $q^*$, for all $t\in \Theta^\psi$ and $x,y\in X$ with $\vartheta_1(x)<t<\vartheta_1(y)$, the inequalities \eqref{help_4} and \eqref{help_5} yield that
 \begin{equation}\label{help_1}
   \frac{\partial_2 \psi(y,t)}{\psi(y,t)}
   \leq -\frac{p'(t)}{p(t)}
   \leq \frac{\partial_2 \psi(x,t)}{\psi(x,t)}.
 \end{equation}
Let $z\in X$ be fixed arbitrarily and let $t\in \Theta^\psi$. If $\vartheta_1(z)<t$, then $\psi(z,t)<0$ and, with $x:=z$, the right hand side of the inequality \eqref{help_1} implies that
 \begin{equation}\label{help_2}
   \frac{\dd}{\dd s}\big(p(s)\psi(z,s)\big)\bigg|_{s=t}
   =p'(t)\psi(z,t)+p(t)\partial_2\psi(z,t)\leq0.
 \end{equation}
If $t<\vartheta_1(z)$, then $\psi(z,t)>0$ and, with $y:=z$, the left hand side of the inequality \eqref{help_1} yields that \eqref{help_2} is also valid.

All in all, we have verified that, for all $z\in X$, the map \eqref{prod_map} is differentiable at the points of the set
$\Theta^\psi\setminus\{\vartheta_1(z)\}$ with a nonpositive derivative and the Lebesgue measure of $\Theta\setminus(\Theta^\psi\setminus\{\vartheta_1(z)\})$ is zero.
On the other hand, for all $z\in X$, the map \eqref{prod_map} is also locally absolutely continuous on $\Theta\setminus\{\vartheta_1(z)\}$,
 since it is the product of two locally absolutely continuous functions (the local absolute continuity of
 $\Theta{\setminus\{\vartheta_1(z)\}}\ni t \mapsto \psi(z,t)$ follows from its continuous differentiability).
Therefore (as a consequence of the Newton--Leibniz formula) we can obtain that, for all $z\in X$, the map \eqref{prod_map} is decreasing and positive on $\Theta\cap(-\infty,\vartheta_1(z))$, and decreasing and negative on $\Theta\cap(\vartheta_1(z),\infty)$.
Indeed, if $z\in X$ and $t_1,t_2\in\Theta\cap(-\infty,\vartheta_1(z))$ with $t_1<t_2$ or $t_1,t_2\in\Theta\cap(\vartheta_1(z),\infty)$ with $t_1<t_2$, then we have
 \begin{align*}
   p(t_2)\psi(z,t_2)
     & = p(t_1)\psi(z,t_1)
         + \int_{t_1}^{t_2} \frac{\dd}{\dd t}\big(p(t)\psi(z,t)\big)\,\dd t\\
     & = p(t_1)\psi(z,t_1)
          + \int_{(t_1,t_2)\cap (\Theta^\psi\setminus\{\vartheta_{1}(z)\}) } \frac{\dd}{\dd t}\big(p(t)\psi(z,t)\big)\,\dd t
       \leq p(t_1)\psi(z,t_1) ,
 \end{align*}
 as desired.
Since $p(t)>0$, $t\in\Theta$, by the definition of $\vartheta_{1}(z)$, we have $p(t)\psi(z,t)>0$ if $t\in\Theta\cap(-\infty,\vartheta_1(z))$, and $p(t)\psi(z,t)<0$ if $t\in\Theta\cap(\vartheta_1(z),\infty)$.
By the assumption $\psi(z,\vartheta_1(z))=0$, $z\in X$, the map \eqref{prod_map} equals $0$ at $t=\vartheta_1(z)$,
 and therefore it is decreasing on the entire interval $\Theta$.
Furthermore, it is easy to see that, for all $z\in X$, the map \eqref{prod_map} is locally absolutely continuous on $\Theta\setminus\{\vartheta_1(z)\}$,
 because it is the product of the locally absolutely continuous function $p:\Theta\to\RR$ and the map $t \mapsto \psi(z,t)$, which, by assumption,
 is continuously differentiable and hence locally absolutely continuous on $\Theta\setminus\{\vartheta_1(z)\}$.
We note that, in general, it can happen that the map \eqref{prod_map} is not continuous at $\vartheta_1(z)$.
We finished the proof of the implication $(i)\Longrightarrow(ii)$.

Next, we assume that assertion (ii) holds, i.e., assume that there exists a positive and locally absolutely continuous function $p:\Theta\to\RR$ such that,
 for all $z\in X$, the map \eqref{prod_map} is decreasing.
Let $x,y\in X$ with $\vartheta_1(x)<\vartheta_1(y)$.
Then, for all $t,s\in\Theta$ with $\vartheta_1(x)<t<s<\vartheta_1(y)$, we have
that
$$
  p(t)\psi(x,t)\geq p(s)\psi(x,s)
  \qquad\mbox{and}\qquad
  p(t)\psi(y,t)\geq p(s)\psi(y,s).
$$
Using that $\psi$ is a $Z_1$-function, we have $\psi(x,s)<0<\psi(y,s)$ and this together with $p(t)>0$ yields that
$$
  \frac{\psi(x,t)}{\psi(x,s)}
  \leq\frac{p(s)}{p(t)}
  \leq\frac{\psi(y,t)}{\psi(y,s)}.
$$
Hence
$$
  \frac{\psi(x,t)}{\psi(x,s)}
  \leq\frac{\psi(y,t)}{\psi(y,s)},
$$
which together with $\psi(y,t)>0$ implies
$$
  -\frac{\psi(x,t)}{\psi(y,t)}
  \leq-\frac{\psi(x,s)}{\psi(y,s)}.
$$
This inequality shows that the map \eqref{function_newhanyados0} is increasing, as desired.

Finally, we note that while proving the implication $(ii)\Longrightarrow(i)$, we did not use the assumptions that $p$ was locally absolutely continuous and that, for all $z\in X$, the map $\Theta\setminus\{\vartheta_1(z)\}\ni t\mapsto \psi(z,t)$ was continuously differentiable.
We also point out the fact that the continuity of the map $\Theta\setminus \{\vartheta_1(z)\}\ni t \mapsto \partial_2\psi(z,t)$ is extensively used proving the implication $(i)\Longrightarrow(ii)$, for example, when we derive that $\Theta\setminus {\{\vartheta_1(z)\} } \ni t\mapsto \psi(z,t)$
is locally absolutely continuous.
\end{proof}

Concerning the verifiability of the conditions of Theorem \ref{Thm_conjecture}, we mention that, if a function $\psi$ is given, then, in many examples, it is easy to check its $Z_1$-property and one can explicitly determine $\vartheta_1$, therefore,
by methods in calculus, one can also check that whether the map $\Theta\setminus\{\vartheta_1(z)\}\ni t\mapsto \psi(z,t)$ is continuously differentiable or not.
The condition that, for all $t\in\Theta$, there exist $x,y\in X$ such that $\vartheta_1(x)<t< \vartheta_1(y)$, is also verifiable provided that we have an expression for $\vartheta_1$ in our hands.

Next, we interpret Theorem~\ref{Thm_conjecture} from a convex optimization theoretical point of view.
Namely, supposing that we are given a family of minimization problems related to continuous functions that are strictly decreasing up to a point
 and strictly increasing from that point on (such functions belong to the set of quasi-convex functions),
 we show that there exists a family of minimization problems related to some appropriate convex functions such that
 the unique solutions of the corresponding members of the two family of minimization problems coincide.

\begin{Thm}\label{Thm_convexity}
Let $X$ be a nonempty set, $\Theta$ be a nondegenerate open interval of $\RR$, and $\varrho:X\times\Theta\to\RR$ be a function.
Assume that, for all $z\in X$, the map $\Theta\ni t\mapsto\varrho(z,t)$ is continuously differentiable and the function $\psi:X\times\Theta\to\RR$ defined by
 \[
  \psi(z,t):=-\partial_2\varrho(z,t)\qquad z\in X,\,t\in \Theta,
 \]
belongs to $\Psi[T,Z_1](X,\Theta)$. In addition, suppose that
\begin{enumerate}[(a)]
 \item for all $t\in\Theta$, there exist $x,y\in X$ such that $\vartheta_1(x)<t< \vartheta_1(y)$,
 \item for all $z\in X$, the map $\Theta\setminus\{\vartheta_1(z)\}\ni t\mapsto\varrho(z,t)$ is twice continuously differentiable.
\end{enumerate}
Then there exists a function $\varrho^*:X\times\Theta\to\RR$ such that
 \begin{enumerate}[(i)]
 \item for all $z\in X$, the map $\Theta\ni t\mapsto\varrho^*(z,t)$ is convex and continuously differentiable,
 \item for all $n\in\NN$ and $x_1,\dots,x_n\in X$,
        \[
           \{\vartheta_n(x_1,\dots,x_n)\}
           =\argmin_{t\in\Theta} \sum_{i=1}^n \varrho(x_i,t)
           =\argmin_{t\in\Theta} \sum_{i=1}^n \varrho^*(x_i,t),
        \]
        that is, $\vartheta_n(x_1,\dots,x_n)$ is the common unique (global) minimum point  in $\Theta$ for the two maps $\Theta \ni t\mapsto \sum_{i=1}^n \varrho(x_i,t)$ and $\Theta \ni t\mapsto \sum_{i=1}^n \varrho^*(x_i,t)$.
 \end{enumerate}
\end{Thm}

\begin{proof}
By the assumption (b), for all $z\in X$, the map $\Theta\setminus\{\vartheta_1(z)\}\ni t\mapsto \psi(z,t)$ is continuously differentiable.
According to part (i) of Theorem~\ref{Thm_M_est_uniq}, for all $x,y\in X$ with $\vartheta_1(x)<\vartheta_1(y)$, the map \eqref{function_newhanyados0} is increasing.
Thus, taking into account assumption (a) as well, we are in the position to apply Theorem~\ref{Thm_conjecture}.
Therefore, there exists a positive and locally absolutely continuous function $p:\Theta\to\RR$ such that, for all $z\in X$, the map \eqref{prod_map}
 is decreasing on $\Theta$ and locally absolutely continuous on $\Theta\setminus\{\vartheta_1(z)\}$.
Define $\varrho^*:X\times\Theta\to\RR$ by
\[
   \varrho^*(z,t)
   :=-\int_{\vartheta_1(z)}^t p(s)\psi(z,s)\,\dd s
   =\int_{\vartheta_1(z)}^t p(s)\partial_2\varrho(z,s)\,\dd s,
   \qquad z\in X,\,t\in\Theta.
\]
By the continuous differentiability of $\varrho$ in its second variable, it follows that $\psi$ is continuous in its second variable.
Using also the continuity of $p$, it follows that $\varrho^*$ is continuously differentiable in its second variable and
 \begin{align}\label{help_10}
 \partial_2\varrho^*(z,t)=-p(t)\psi(z,t), \qquad z\in X,\,t\in\Theta.
 \end{align}
Applying the decreasingness of the map \eqref{prod_map} on $\Theta$, we can see that $\partial_2\varrho^*$ is increasing on $\Theta$ in its second variable.
Consequently, $\varrho^*$ is convex on $\Theta$ in its second variable. This proves assertion \textit{(i)}.

To see that assertion \textit{(ii)} is valid, let $n\in\NN$ and $x_1,\dots,x_n\in X$ be fixed. The derivative of the map $\Theta\ni t\mapsto\sum_{i=1}^n \varrho(x_i,t)$ has the following form
\[
  \Theta\ni t\mapsto\sum_{i=1}^n \partial_2\varrho(x_i,t)
  =-\sum_{i=1}^n \psi(x_i,t).
\]
By the property $[T_n]$ of $\psi$, we have
\Eq{sign_prop}{
  \sum_{i=1}^n \psi(x_i,t)
  \begin{cases}
   >0 &\mbox{if } t<\vartheta_n(x_1,\dots,x_n), \\
   <0 &\mbox{if } t>\vartheta_n(x_1,\dots,x_n),
  \end{cases}
  \qquad t\in\Theta.
}
Therefore,
\Eq{*}{
  \Theta\ni t\mapsto\sum_{i=1}^n \partial_2\varrho(x_i,t)
  \begin{cases}
   <0 &\mbox{if } t<\vartheta_n(x_1,\dots,x_n), \\
   >0 &\mbox{if } t>\vartheta_n(x_1,\dots,x_n),
  \end{cases}
  \qquad t\in\Theta,
}
which shows that the continuous map $\Theta\ni t\mapsto\sum_{i=1}^n \varrho(x_i,t)$ is strictly decreasing on $\Theta\cap(-\infty,\vartheta_n(x_1,\dots,x_n))$ and strictly increasing on $\Theta\cap(\vartheta_n(x_1,\dots,x_n),\infty)$. This yields that
\[
  \argmin_{t\in\Theta} \sum_{i=1}^n \varrho(x_i,t)
  =\{\vartheta_n(x_1,\dots,x_n)\}.
\]
On the other hand, using \eqref{help_10}, the derivative of the map $\Theta\ni t\mapsto\sum_{i=1}^n \varrho^*(x_i,t)$ takes the form
\[
  \Theta\ni t\mapsto\sum_{i=1}^n \partial_2\varrho^*(x_i,t)
  =-\sum_{i=1}^n p(t)\psi(x_i,t).
\]
Applying \eqref{sign_prop} and the positivity of $p$, it follows that
\Eq{*}{
  \Theta\ni t\mapsto\sum_{i=1}^n \partial_2\varrho^*(x_i,t)
  \begin{cases}
   <0 &\mbox{if } t<\vartheta_n(x_1,\dots,x_n), \\
   >0 &\mbox{if } t>\vartheta_n(x_1,\dots,x_n),
  \end{cases}
  \qquad t\in\Theta.
}
As above, this implies that
\[
  \argmin_{t\in\Theta} \sum_{i=1}^n \varrho^*(x_i,t)
  =\{\vartheta_n(x_1,\dots,x_n)\}.
\]
Thus, we have completed the proof of assertion \textit{(ii)}.
\end{proof}

Concerning the verifiability of the conditions of Theorem \ref{Thm_convexity}, we mention that, given a function $\varrho$, the continuously differentiability of the map $\Theta\ni t\mapsto\varrho(z,t)$ can be checked in a standard manner.
Then one can calculate $\psi(z,t):=-\partial_2\varrho(z,t)$, $z\in X$, $t\in \Theta$.
The condition that $\psi\in \Psi[Z_1](X,\Theta)$ is verifiable, one just has to solve the equation $\partial_2\varrho(z,t) = 0$ in terms of $t\in\Theta$,
 where $z\in X$ is fixed. 
By this procedure, one can also get an expression for the function $\vartheta_1$.
The conditions (a) and (b) of Theorem \ref{Thm_convexity} can be checked by common methods of calculus.
The most difficult task could be verifying the condition that $\psi\in \Psi[T](X,\Theta)$.
In many examples, see Section 4 of Barczy and P\'ales \cite{BarPal2} and Section 3 of Barczy and P\'ales \cite{BarPal3}, we presented (generalized) $\psi$-estimators, where one can check this condition and explicitly calculate the function $\vartheta_n$ as well.
See also Examples \ref{Ex_mle} and \ref{Ex_pareto} of the present paper.

Next, we demonstrate Theorem \ref{Thm_conjecture} and the construction of the function $p$ introduced in \eqref{formula_p} (detailed in the proof of Theorem \ref{Thm_conjecture}) by giving two examples from the theory of maximum likelihood estimation.

\begin{Ex}\label{Ex_mle}
Let $\xi$ be a normally distributed random variable with mean $m\in\RR$ and with variance $\sigma^2$, where $\sigma>0$.
Let $n\in\NN$ and $x_1,\ldots,x_n\in\RR$ be a realization of a sample of size $n$ for $\xi$.
Assume that $m$ is known.
Then there exists a unique maximum likelihood estimator (MLE) of $\sigma^2$ based on $x_1,\ldots,x_n\in\RR$,
 and it takes the form \ $\widehat{\sigma^2_n}:=\frac{1}{n}\sum_{i=1}^n (x_i-m)^2$.
In the second part of Example 9 in Barczy and P\'ales \cite{BarPal2}, we have shown the existence and uniqueness of a solution of the corresponding likelihood equation using Theorem 1 in \cite{BarPal2}.
It turned out that one can choose $X=\RR$, $\Theta=(0,\infty)$, and $\psi:\RR\times(0,\infty)\to\RR$,
 \begin{align*}
    \psi(x,\sigma^2)
      = \frac{1}{2(\sigma^2)^2}\left( (x-m)^2 - \sigma^2\right) ,
           \qquad x\in\RR,\; \sigma^2>0.
 \end{align*}
Hence $\psi$ is a $T_1$-function with $\vartheta_1(x):=(x-m)^2$, $x\in\RR$, $\psi(x,\vartheta_1(x))=0$, $x\in\RR$, yielding that $\psi$ is a $Z_1$-function, for all $\sigma^2\in(0,\infty)$, there exist $x,y\in \RR$ such that $\vartheta_1(x) = (x-m)^2 < \sigma^2 < (y-m)^2 = \vartheta_1(y)$, and $\psi$ is continuously differentiable in its second variable.
Further, for all $x,y\in\RR$ with $\vartheta_1(x)< \vartheta_1(y)$, i.e., $(x-m)^2 < (y-m)^2$, we get that the function
 \eqref{function_newhanyados0} takes the form
 \begin{align*}
  \big( (x-m)^2, (y-m)^2\big)\ni\sigma^2
     \mapsto -\frac{\psi(x,\sigma^2)}{\psi(y,\sigma^2)}
       = -1 + \frac{(y-m)^2 - (x-m)^2 }{(y-m)^2 - \sigma^2},
 \end{align*}
 which is strictly increasing.
In what follows, we determine the functions $q_*$, $q^*$ and $p$ introduced in \eqref{formulae_qs} and \eqref{formula_p}, respectively.
Since
 \[
   \partial_2\psi(x,\sigma^2) = \frac{(\sigma^2)^2 - 2 \sigma^2 (x-m)^2}{2(\sigma^2)^4}
      = \frac{\sigma^2 - 2(x-m)^2}{2(\sigma^2)^3}, \qquad x\in\RR,\; \sigma^2>0,
 \]
we have
 \[
   \frac{\partial_2\psi(x,\sigma^2)}{\psi(x,\sigma^2)}
       = \frac{\sigma^2 - 2(x-m)^2}{\sigma^2( (x-m)^2 - \sigma^2)}
       = -\frac{2}{\sigma^2} - \frac{1}{(x-m)^2 - \sigma^2}
 \]
 for $x\in\RR$ and $\sigma^2>0$ satisfying $\sigma^2\ne (x-m)^2$.
Consequently, for all $\sigma^2>0$, we get
 \begin{align*}
  q_*(\sigma^2)& =\sup\bigg\{\frac{\partial_2 \psi(y,\sigma^2)}{\psi(y,\sigma^2)}\,\bigg|\,y\in \RR,\,\sigma^2 < (y-m)^2 \bigg\}\\
               & = \sup\bigg\{ -\frac{2}{\sigma^2} - \frac{1}{z-\sigma^2} \,\bigg|\, z\in\RR, \, \sigma^2 < z \bigg\}
                 = \sup\bigg(-\infty, - \frac{2}{\sigma^2}\bigg) = - \frac{2}{\sigma^2},
 \end{align*}
 and
 \begin{align*}
  q^*(\sigma^2)&=\inf\bigg\{\frac{\partial_2 \psi(x,\sigma^2)}{\psi(x,\sigma^2)}\,\bigg|\,x\in \RR,\, (x-m)^2<\sigma^2\bigg\}\\
               & = \inf\bigg\{ -\frac{2}{\sigma^2} - \frac{1}{z-\sigma^2} \,\bigg|\, z\in\RR, \, z < \sigma^2 \bigg\}
                 = \inf \bigg(- \frac{2}{\sigma^2},\infty\bigg) = - \frac{2}{\sigma^2}.
 \end{align*}
Hence, in this case, $q_*=q^*$, and $q_*:(0,\infty)\to\RR$, $q_*(s) = -\frac{2}{s}$, $s>0$.
Further, for an arbitrarily fixed $\sigma_0^2\in\RR_{++}$, the function $p:(0,\infty)\to\RR_{++}$ introduced in \eqref{formula_p} takes the form
 \begin{align*}
  p(\sigma^2) = \exp\bigg(-\int_{\sigma_0^2}^{\sigma^2} q_*(s)\,\dd s\bigg)
              = \exp\bigg(\int_{\sigma_0^2}^{\sigma^2} \frac{2}{s} \,\dd s\bigg)
              = \exp\big(2(\ln(\sigma^2) - \ln(\sigma_0^2))\big)
              = \left(\frac{\sigma^2}{\sigma_0^2}\right)^2
 \end{align*}
 for $\sigma^2>0$.
Hence, we can see that $p$ is positive, locally absolutely continuous, and, for all $x\in\RR$, the map $(0,\infty)\ni\sigma^2 \mapsto  p(\sigma^2)\psi(x,\sigma^2) = \frac{1}{2(\sigma_0^2)^2} ( (x-m)^2 - \sigma^2)$ is (strictly) decreasing, which is in accordance with Theorem \ref{Thm_conjecture}.
\proofend
\end{Ex}

\begin{Ex}\label{Ex_pareto}
Let $\xi$ be a random variable having Pareto distribution with scale parameter $1$ and shape parameter $\alpha>0$,
 i.e., $\xi$ has a density function
 \[
     f_\xi(x) := \begin{cases}
                  \alpha x^{-(\alpha+1)} & \text{if $x>1$,}\\[2mm]
                  0 & \text{if $x\leq 1$.}
               \end{cases}
 \]
First, we will establish the existence and uniqueness of a solution of the likelihood equation for $\alpha$ using Theorem 1 and 
 Proposition 2 in Barczy and P\'ales \cite{BarPal2}.
In the considered case, using the setup given before Example 9 in \cite{BarPal2}, we have $\Theta := (0,\infty)$ and $f:\RR\times (0,\infty)\to\RR$,
 \[
     f(x,\alpha) :=\begin{cases}
                  \alpha x^{-(\alpha+1)} & \text{if $x>1$, $\alpha>0$,}\\[2mm]
                  0 & \text{otherwise,}
                \end{cases}
 \]
and consequently, $\cX_f=(1,\infty)$.
Then \ $\psi:(1,\infty)\times (0,\infty)\to\RR$,
 \begin{align*}
   \psi(x,\alpha) = \partial_2 \big(\ln(f(x,\alpha))\big)
                  = \partial_2\left( \ln(\alpha) - (\alpha+1)\ln(x) \right) 
                  =  \frac{1}{\alpha} - \ln(x), \qquad  x\in(1,\infty),\;\alpha\in(0,\infty).
 \end{align*}
Then $\psi$ is continuously differentiable in its second variable, and $\psi$ is a $Z_1$-function with $\vartheta_1(x) = 1/\ln(x)$, $x\in(1,\infty)$,
 yielding that, for all $\alpha\in(0,\infty)$, there exist $x,y\in(1,\infty)$ such that 
 \[
   \vartheta_1(x) = \frac{1}{\ln(x)} < \alpha <  \frac{1}{\ln(y)} = \vartheta_1(y).
 \]  
This, together with the fact that $\psi$ is strictly decreasing and continuous in its second variable, yields that $\psi$ is a $Z$-function,
 see Proposition 2 in \cite{BarPal2}.
Further, for all $n\in\NN$ and $x_1,\ldots,x_n\in(1,\infty)$, the likelihood equation for $\alpha$ takes the form
 \[
   \sum_{i=1}^n \psi(x_i,\alpha) = \frac{n}{\alpha} - \sum_{i=1}^n \ln(x_i)=0,\qquad \alpha>0,
 \]
 which has a unique solution
 \[
  \vartheta_n(x_1,\ldots,x_n) = \frac{1}{\frac{1}{n}\sum_{i=1}^n \ln(x_i)}.
 \]
 
We call the attention that, in the present example, for all $x\in(1,\infty)$, 
 the function $(0,\infty)\ni\alpha\mapsto \psi(x,\alpha)=\frac{1}{\alpha} - \ln(x)$ is (strictly) decreasing and continuously differentiable 
 (and, hence, locally absolutely continuous).
Consequently, the function \eqref{prod_map} in part (ii) of Theorem \ref{Thm_conjecture} 
 is decreasing and continuously differentiable (and, hence, locally absolutely continuous)
 for any choice of $p(t):=c$, $t\in(0,\infty)$, where $c$ is arbitrary positive constant.  
Nonetheless, in what follows, 
 we determine the functions $q_*$, $q^*$ and $p$ introduced in \eqref{formulae_qs} and \eqref{formula_p}, respectively,
 in order to demonstrate Theorem \ref{Thm_conjecture}.

For all $x,y\in(1,\infty)$ with $\vartheta_1(x)< \vartheta_1(y)$, i.e., $y < x$, we get that the function
 \eqref{function_newhanyados0} takes the form
 \begin{align*}
  \left( \frac{1}{\ln(x)}, \frac{1}{\ln(y)} \right)\ni \alpha
       \mapsto -\frac{\psi(x,\alpha)}{\psi(y,\alpha)}
        = - \frac{1-\alpha\ln(x)}{1-\alpha\ln(y)}
        = -1 + \frac{\ln(x) - \ln(y)}{\frac{1}{\alpha} - \ln(y)},
 \end{align*}
 which is strictly increasing.
Since
 \[
   \partial_2\psi(x,\alpha) = -\frac{1}{\alpha^2}, \qquad x>1,\; \alpha>0,
 \]
we have
 \[
   \frac{\partial_2\psi(x,\alpha)}{\psi(x,\alpha)}
       = \frac{ -\frac{1}{\alpha^2} }{ \frac{1}{\alpha} - \ln(x) }
       = -\frac{1}{ \alpha - \alpha^2 \ln(x)},
       \qquad x>1,\; \alpha>0.
 \]
Consequently, for all $\alpha>0$, we get
 \begin{align*}
  q_*(\alpha)& =\sup\bigg\{ -\frac{1}{ \alpha - \alpha^2 \ln(y)}  \,\bigg|\,y>1,\,\alpha < \frac{1}{\ln(y)} \bigg\}\\
               & = \sup\bigg\{ -\frac{1}{ \alpha - \alpha^2 \ln(y)} \,\bigg|\, y \in\big(1,\ee^{\frac{1}{\alpha}}\big) \bigg\}
                 = \sup\bigg(-\infty, - \frac{1}{\alpha}\bigg) = - \frac{1}{\alpha},
 \end{align*}
 and
 \begin{align*}
  q^*(\alpha)&=\inf\bigg\{ -\frac{1}{ \alpha - \alpha^2 \ln(x)}  \,\bigg|\,x>1,\,  \frac{1}{\ln(x)} < \alpha \bigg\}\\
               & = \inf\bigg\{ -\frac{1}{ \alpha - \alpha^2 \ln(x)}\,\bigg|\,  x > \ee^{\frac{1}{\alpha}}   \bigg\}
                 = \inf \big(0,\infty\big) = 0.
 \end{align*}
Hence, in this case, $q_*<q^*$, and $q_*:(0,\infty)\to\RR$, $q_*(s) = -\frac{1}{s}$, $s>0$.
Further, for an arbitrarily fixed $\alpha_0>0$, the function $p:(0,\infty)\to\RR_{++}$ introduced in \eqref{formula_p} takes the form
 \begin{align*}
  p(\alpha) = \exp\bigg(-\int_{\alpha_0}^{\alpha} q_*(s)\,\dd s\bigg)
              = \exp\bigg(\int_{\alpha_0}^{\alpha} \frac{1}{s} \,\dd s\bigg)
              = \exp\big(\ln(\alpha) -  \ln(\alpha_0)\big)
              = \frac{\alpha}{\alpha_0},
              \qquad \alpha>0.
 \end{align*}
Therefore, we can see that $p$ is positive, locally absolutely continuous, and, for all $x>1$, 
 the map $(0,\infty)\ni \alpha \mapsto  p(\alpha)\psi(x,\alpha) 
    = \frac{\alpha}{\alpha_0} \left(\frac{1}{\alpha} - \ln(x) \right) 
    = \frac{1}{\alpha_0}(1-\alpha\ln(x))$ is (strictly) decreasing, which is in accordance with Theorem \ref{Thm_conjecture}. 
\proofend
\end{Ex}

\section{Measurability of generalized $\psi$-estimators}
\label{Sec_measurability}

In Barczy and P\'ales \cite[Proposition 1]{BarPal2}, we proved that if $(X,\cX)$ is a measurable space, $n\in\NN$, $\psi\in\Psi[Z_n](X,\Theta)$, and $\psi$ is measurable in its first variable (i.e., for all $\vartheta\in\Theta$, the mapping $X\ni x\mapsto \Psi(x,\vartheta)$ is measurable), then $\vartheta_{n,\psi}:X^n\to \Theta$ is measurable with respect to the $n$-fold product sigma-algebra $\cX^n$ and the Borel sigma-algebra on $\Theta$.
We are going to prove that, under some additional assumptions on $X$ and $\psi$, a converse of this statement holds as well.
These results are important from theoretical point of view in the sense that, roughly speaking, together they state that the measurability of a generalized $\psi$-estimator is equivalent to the measurability of the corresponding function $\psi$ in its first variable provided that the underlying measurable space $X$ has a measurable diagonal, $\psi$ has the property $[Z]$ and fulfills another additional assumption.

First, we recall some notions and results on measurable spaces, which will play an important role in our proof.
A measurable space $(X,\cX)$ is said to have a measurable diagonal if the set $\{(x_1,x_2)\in X\times X : x_1=x_2\}$ belongs to the $2$-fold product sigma-algebra $\cX^2:=\cX\otimes\cX$.
We say that a family $\cE$ of subsets of a nonempty set $X$ separates the points of $X$ if for all two distinct points $x,y\in X$, there exists a set $E\in \cE$ such that either $x\in E$ and $y\notin E$ or $y\in E$ and $x\notin E$.
We say that a sigma-algebra $\cX$ on $X$ is countably separated if there exists an at most countable family $\cE$ of $\cX$ which separates the points of $X$.
We say that a sigma-algebra on $X$ is countably generated (or separable) if there exists an at most countable family $\cE$ of $\cX$, which generates $\cX$.
For example, the Borel sigma-algebra of a separable metric space is always countably separated and countably generated.
Indeed, a separable metric space has a countable dense subset and one can take $\cE$ to be the family of open balls around the points of this dense subset with rational radii.
Given a measurable space $(X,\cX)$ and a non-empty subset $A$ of $X$, by the trace of $\cX$ on $A$, we mean the family $\{B\cap A : B\in\cX\}$.

For the following result on measurable spaces with measurable diagonals, see Bogachev \cite[Theorem 6.5.7]{Bog}.

\begin{Thm}[Bogachev \cite{Bog}]\label{Thm_Bog2}
Let $(X,\cX)$ be a measurable space.
The following statements are equivalent:
 \begin{itemize}
   \item[(i)] $\cX$ is a countably separated sigma-algebra,
   \item[(ii)] there exists an injective measurable function $f:X\to [0,1]$,
   \item[(iii)] $(X,\cX)$ has a measurable diagonal, i.e., $\{(x,x): x\in X\}\in \cX\otimes\cX$,
   \item[(iv)] there exists a separable sigma-algebra $\cX_0$ such that $\cX_0\subset \cX$ and all the singletons $\{x\}$, $x\in X$, belong to $\cX_0$.
 \end{itemize}
\end{Thm}

Next, we formulate a consequence of Theorem 6.5.7 in Bogachev \cite{Bog}.

\begin{Lem}\label{Lem_nfold_diag_measurable}
Let $(X,\cX)$ be a measurable space such that there is a countable subset of $\cX$, which separates the points of $X$ (equivalently, $(X,\cX)$ has a measurable diagonal, see Theorem \ref{Thm_Bog2}).
Then, for each $n\in\NN$, the set
 \begin{align*}
    \Big\{ (x_1,\ldots,x_n)\in X^n : x_1=\cdots = x_n \Big\}
 \end{align*}
is in the $n$-fold product sigma-algebra $\cX^n$.
\end{Lem}

\begin{proof}
For each $n\in\NN$, let us introduce the notation
 \begin{align}\label{Def_Diag_n}
   \Diag_n X:=\Big\{ (\underbrace{x,\ldots, x}_n) : x\in X \Big\}.
 \end{align}
For $n=2$, the statement follows from the equivalence of (i) and (iii) of Theorem 6.5.7 in Bogachev \cite{Bog} (see also Theorem \ref{Thm_Bog2}).
For each $n\in\NN$, $n\geq3$, we have that
 \begin{align}\label{help_7}
    \Diag_n X
       &= \bigcap_{i=1}^{n-1} \Big\{ (x_1,\ldots, x_n)\in X^n : x_i=x_{i+1} \Big\}\\\nonumber
       &= \big((\Diag_2 X) \times X^{n-2}\big) \cap \bigg(\bigcap_{i=2}^{n-2} X^{i-1}\times (\Diag_2 X) \times X^{n-i-1}\bigg)
       \cap \big(X^{n-2}\times (\Diag_2 X)\big).
 \end{align}
By the equivalence of (i) and (iii) of Theorem 6.5.7 in Bogachev \cite{Bog}, we have that $\Diag_2 X$ belongs to the $2$-fold product sigma algebra $\cX^2$.
Using that the $n$-fold product sigma-algebra $\cX^n$ is generated by the family of all the sets
 $A_1\times \cdots\times A_n$ with $A_i\in\cX$, $i\in\{1,\ldots,n\}$, by Klenke \cite[part (i) of Theorem 14.12]{Kle}, we have that
 \[
   \cX^2\otimes \cX^{n-2}\subseteq \cX^n,\qquad
   \cX^{i-1}\otimes \cX^2\otimes \cX^{n-i-1}\subseteq \cX^n,\quad i\in\{2,\ldots,n-2\}, \qquad  \cX^{n-2}\otimes \cX^2\subseteq \cX^n.
 \]
Since
 \begin{align*}
    (\Diag_2 X) \times X^{n-2}&\in \cX^2\otimes \cX^{n-2}\subseteq \cX^n,\\
    X^{i-1}\times (\Diag_2 X) \times X^{n-i-1} &\in \cX^{i-1}\otimes \cX^2\otimes \cX^{n-i-1} \subseteq \cX^n,\qquad i\in\{2,\ldots,n-2\},\\
    X^{n-2}\times (\Diag_2 X)&\in \cX^{n-2}\otimes \cX^2\subseteq \cX^n,
 \end{align*}
 using \eqref{help_7}, we arrive at $\Diag_n X\in \cX^n$, as desired.
\end{proof}

\begin{Thm}\label{Thm_measurability}
Let $(X,\cX)$ be a measurable space such that there is a countable subset of $\cX$, which separates the points of $X$ (equivalently, $(X,\cX)$ has a measurable diagonal, see Theorem \ref{Thm_Bog2}).
Let $\psi\in\Psi[Z](X,\Theta)$ be such that, for all $r\in\Theta$, the function $X\ni y\mapsto \psi(y,r)$ takes positive and negative values as well.
Suppose that, for each $n\in\NN$, the map $\vartheta_{n,\psi}:X^n\to \Theta$ is measurable with respect to
 the $n$-fold product sigma-algebra $\cX^n$ and the Borel sigma-algebra on $\Theta$.
Then $\psi$ is measurable in its first variable.
\end{Thm}

\begin{proof}
For each $n\in\NN$, let us furnish the set $\Diag_n X$ (introduced in \eqref{Def_Diag_n}) with the sigma-algebra $\cS_n:= \cX^n\cap \Diag_n X$, i.e., $\cS_n$ is the trace of $\cX^n$ on $\Diag_n X$.
Then $\cS_n$ is indeed a sigma-algebra on $\Diag_n X$, since $\Diag_n X\in \cS_n$; $\cS_n$ is closed under
the set-theoretic complement due to the fact that $A\setminus (B\cap A) = (A\setminus B)\cap A$ for any sets $A$ and $B$; and $\cS_n$ is trivially closed under countable unions.
Consequently, $(\Diag_n X, \cS_n)$ is a measurable space for each $n\in\NN$.

For each $n\in\NN$, let the projection $\pi_n:\Diag_n X\to X$ be given by
\Eq{*}{
  \pi_n(\underbrace{x,\ldots, x}_{n}):=x, \qquad x\in X.
}
Then $\pi_n$ is a bijection between $\Diag_n X$ and $X$.
We are going show that $\pi_n$ and its inverse $\pi_n^{-1}$ are measurable, i.e., $\pi_n$ is bimeasurable.
For all $A\in\cX$, we have that
 \begin{align*}
  \pi_n^{-1}(A)
  & =\Big\{ (\underbrace{x,\ldots, x}_{n}) \in \Diag_n X : \pi_n( \underbrace{x,\ldots, x}_{n} )\in A\Big\} \\
  & =\Big\{ (\underbrace{x,\dots, x}_{n}) \in \Diag_n X : x\in A\Big\} \\
  &= A^n\cap\Diag_n X \in\cX^n\cap\Diag_n X=\cS_n.
 \end{align*}
Therefore, for each $n\in\NN$, $\pi_n$ is measurable.
To prove that $\pi_n^{-1}$ is measurable,
we need to show that $(\pi_n^{-1})^{-1}(B) = \pi_n(B)\in\cX$ for all $B\in\cS_n$.
Using that the trace of a sigma-algebra generated by a family of subsets (generating system) on a nonempty subset coincides
 with the generated sigma-algebra of the trace of the generating system on the given nonempty subset (see, e.g., Klenke \cite[Corollary 1.83]{Kle}), we have that the family
 \[
  \big\{(A_1\times\cdots\times A_n)\cap\Diag_n X\mid A_1,\dots,A_n\in\cX\big\}
 \]
forms a generating system of the sigma-algebra $\cS_n$.
Consequently, to show the measurability of $\pi_n^{-1}$, it is enough to check that
 $\pi_n((A_1\times\cdots\times A_n)\cap\Diag_n X)\in\cX$ (see, e.g., Klenke \cite[Theorem 1.81]{Kle}).
For all $A_1,\dots,A_n\in\cX$, we have that
$$
\pi_n((A_1\times\cdots\times A_n)\cap\Diag_n X)
=\{x\in X\mid (x,\dots,x)\in A_1\times\cdots\times A_n\}
=A_1\cap\cdots\cap A_n\in\cX.
$$
This proves that $\pi_n^{-1}$ is measurable, henceforth, $\pi_n$ is bimeasurable.

For each $k,m\in \NN$, let us introduce the projection $\pi_k\otimes\pi_m: \Diag_k X\times\Diag_m X \to X^2$,
 \[
    (\pi_k\otimes\pi_m)(u,v):=(\pi_k(u),\pi_m(v)),
    \qquad  (u,v)\in \Diag_k X\times\Diag_m X.
 \]
We check that, for each $k,m\in \NN$, $\pi_k\otimes\pi_m$ is a bimeasurable bijection with respect to the sigma-algebras $\cS^k\otimes\cS^m$ and $\cX^2$.
To show the measurability of $\pi_k\otimes\pi_m$, it is enough to check that
 \[
  (\pi_k\otimes \pi_m)^{-1}(A_1\times A_2)\in \cS^k\otimes\cS^m, \qquad A_1,A_2\in\cX.
 \]
Since
 \begin{align*}
    (\pi_k\otimes \pi_m)^{-1}(A_1\times A_2)
       & = \Big\{ (u,v)\in \Diag_k X\times\Diag_m X : (\pi_k\otimes \pi_m)(u,v)\in A_1\times A_2   \Big\}\\
       & = \Big\{ (u,v)\in \Diag_k X\times\Diag_m X : \pi_k(u)\in A_1, \pi_m(v)\in A_2   \Big\}  \\
       & = \Big\{ u \in \Diag_k X: \pi_k(u)\in A_1 \Big\}
           \times \Big\{ v \in \Diag_m X: \pi_m(v)\in A_2 \Big\}\\
       & = \pi_k^{-1}(A_1) \times \pi_m^{-1}(A_2)
         \in \cS_k\otimes \cS_m,
 \end{align*}
where, at the last step, we used that $\pi_n$ is a measurable bijection for each $n\in\NN$.
To show the measurability of $(\pi_k\otimes\pi_m)^{-1}$, it is enough to check that
 \[
    ((\pi_k\otimes\pi_m)^{-1})^{-1}(B_1\times B_2) = (\pi_k\otimes\pi_m)(B_1\times B_2)\in\cX^2,
        \qquad B_1\in\cS_k,\; B_2\in\cS_m.
 \]
Since
 \begin{align*}
    (\pi_k\otimes \pi_m)(B_1\times B_2)
       & = \Big\{ (\pi_k\otimes \pi_m)(u,v): u\in B_1, v\in B_2 \Big\}\\
       & = \Big\{  (\pi_k(u),\pi_m(v))  :  u\in B_1, v\in B_2  \Big\}  \\
       & = \Big\{ \pi_k(u) : u\in B_1 \Big\}
           \times \Big\{ \pi_m(v) : v\in B_2 \Big\}\\
       & = \pi_k(B_1) \times \pi_m(B_2)
         \in \cX\otimes \cX = \cX^2,
 \end{align*}
where, at the last step, we used that $\pi_n^{-1}$ is a measurable bijection for each $n\in\NN$.

For any function $\psi\in\Psi(X,\Theta)$ having the property $[T]$ and for all $r\in\Theta$ and $n\in\NN$,
 it holds that
 \begin{align}\label{help_9}
  \begin{split}
   \vartheta_{n,\psi}^{-1}((-\infty, r))
       & = \big\{ (x_1,\dots,x_n)\in X^n : \vartheta_{n,\psi}(x_1,\ldots,x_n)<r \big\} \\
       & \subseteq \big\{(x_1,\dots,x_n)\in X^n :  \psi(x_1,r)+\dots+\psi(x_n,r)<0\big\}.
  \end{split}
 \end{align}
Using that the function $\psi$ has the property $[Z]$ as well, the reverse inclusion in \eqref{help_9} holds as well, i.e., the inclusion in \eqref{help_9} is in fact an equality.
This, together with the assumption that the mapping $\vartheta_{n,\psi}:X^n\to \Theta$ is measurable, yields that
 \begin{align}\label{help_8}
  \{(x_1,\dots,x_n)\in X^n :  \psi(x_1,r)+\dots+\psi(x_n,r)<0\}
    = \vartheta_{n,\psi}^{-1}((-\infty, r))\in\cX^n
 \end{align}
for all $r\in\Theta$ and $n\in\NN$.

Using that there is a countable subset of $\cX$ that separates the points of $X$, by Lemma \ref{Lem_nfold_diag_measurable}, we also have $\Diag_n X \in\cX^n$ for each $n\in\NN$.
Hence, for each $k,m\in\NN$, the set
 \begin{align*}
   &\big\{  (x_1,\ldots,x_{k+m})\in X^{k+m} : x_1=\cdots=x_k, \; x_{k+1}=\cdots=x_{k+m} \big\} \\
   &\quad  = \big\{  (x_1,\ldots,x_{k+m})\in X^{k+m} : x_1=\cdots=x_k\big\}
        \cap \big\{  (x_1,\ldots,x_{k+m})\in X^{k+m} : x_{k+1}=\cdots=x_{k+m} \big\}\\
   &\quad = \big((\Diag_k X) \times X^m  \big) \cap \big( X^k \times (\Diag_m X)  \big)
   = (\Diag_k X) \times (\Diag_m X)
 \end{align*}
belongs to $\cX^k\otimes\cX^m = \cX^{k+m}$ (this equality follows, e.g., from part (i) of Theorem 14.12 in Klenke \cite{Kle}, namely,
 $\cX^{k+m}$ is generated by the sets of the form $A\times B$, where $A\in\cX^k$ and $B\in\cX^m$).
As a consequence, using \eqref{help_8} and that a sigma-algebra is closed under intersection, for each $k,m\in\NN$ and $r\in\Theta$, we have that the set
 \begin{align*}
  S_{k,m,r}:=\big\{  (x_1,\ldots,x_{k+m})\in X^{k+m} &: \phantom{:\;\;} \psi(x_1,r)+\cdots+\psi(x_{k+m},r) < 0 \big\}\cap\big((\Diag_k X) \times (\Diag_m X)\big) \\
  =\big\{  (x_1,\ldots,x_{k+m})\in X^{k+m} &: x_1=\cdots=x_k, \; x_{k+1}=\cdots=x_{k+m},\\
    &\phantom{:\;\;} \psi(x_1,r)+\cdots+\psi(x_{k+m},r) < 0 \big\}\\
  =\big\{  (x_1,\ldots,x_{k+m})\in X^{k+m} &: x_1=\cdots=x_k, \; x_{k+1}=\cdots=x_{k+m},
  \\&\phantom{:\;\;}
  k\psi(x_1,r)+m\psi(x_{k+1},r) < 0 \big\}
 \end{align*}
belongs to $\cX^{k+m}$.
Note also that $S_{k,m,r}\subseteq(\Diag_k X)\times(\Diag_m X)\subseteq X^k\times X^m$.
Therefore, we obtain that
 \[
    S_{k,m,r} \in (\cX^k\otimes\cX^m) \cap ((\Diag_k X)\times(\Diag_m X))
               = \cS_k\otimes\cS_m,
 \]
where at the equality we used again that the trace of a sigma-algebra generated by a family of subsets (generating system)
on a nonempty subset coincides with the generated sigma-algebra of the trace of the generating system on the given nonempty subset (see, e.g., Klenke \cite[Corollary 1.83]{Kle}).

Note that, for all $r\in\Theta$ and for each $k,m\in\NN$, we have that
 \begin{align*}
  \Big\{ (x,y)\in X^2 :  \psi(x,r) + \frac{m}{k}\psi(y,r)<0\Big\}
    &= \{ (x,y)\in X^2 :  k\psi(x,r) + m\psi(y,r)<0\} \\
    &= (\pi_k\otimes\pi_m)(S_{k,m,r}).
  \end{align*}
Since $S_{k,m,r}\in \cS_k\otimes\cS_m$ and $\pi_k\otimes\pi_m$ is a bimeasurable bijection,
we have that $(\pi_k\otimes\pi_m)(S_{k,m,r})\in\cX^2$.
Consequently (replacing $\frac{m}{k}$ by $\lambda$),
for all $r\in\Theta$ and all rational numbers $\lambda>0$, we get
 \begin{align*}
  \big\{ (x,y)\in X^2 :  \psi(x,r)<-\lambda\psi(y,r) \big\}\in\cX^2.
 \end{align*}
Therefore, using that the sections of a $\cX^2$-measurable set are $\cX$-measurable sets (see, e.g., Bogachev \cite[Proposition 3.3.2]{Bog}), for all $r\in\Theta$, for all rational numbers $\lambda>0$ and for all $y\in X$, we get that
 \[
  \{x\in X :  \psi(x,r)<-\lambda\psi(y,r)\} \in\cX.
 \]
By the assumption that, for all $r\in\Theta$, the function $X\ni y\mapsto \psi(y,r)$ takes positive and negative values as well, it follows that there exists a sequence $(c_n)_{n\in\NN}$ of real numbers such that the set $\{c_n : n\in\NN\}$ is dense in $\RR$, and
\Eq{*}{
  \{x\in X :  \psi(x,r)<c_n\}\in\cX, \qquad r\in\Theta,\quad n\in\NN.
}
Let $c\in\RR$ be arbitrary and choose a subsequence $(c_{n_k})_{k\in\NN}$ of $(c_n)_{n\in\NN}$ such that $c-\frac{1}{k}<c_{n_k}<c$, $k\in\NN$.
Then $c_{n_k}\uparrow c$ as $k\to\infty$ and
\Eq{*}{
  \{x\in X :  \psi(x,r)<c\}
  =\bigcup_{k\in\NN}\{x\in X : \psi(x,r)<c_{n_k}\}\in\cX,
 }
where we used that a sigma-algebra is closed under countable unions.
\end{proof}

\section{Concluding remarks}

In this paper, we have investigated monotone representation and measurability properties 
 of generalized $\psi$-estimators introduced by Barczy and P\'ales \cite{BarPal2}.
Such estimators are generalizations of usual $\psi$-estimators in the sense that 
  we are searching for a point of sign change of an appropriate function instead of a zero.
Our first main result, roughly speaking, states that, given a generalized $\psi$-estimator with some 
 regularity assumptions on $\psi$, one can construct this estimator in terms of 
 a modified version of the function $\psi$, which is decreasing in its second variable. 
We highlight our construction in question by investigating the maximum likelihood
 estimator of the mean of a normally distributed random variable with a given variance, 
 and that of the shape parameter of a random variable with Pareto distribution having scale parameter 1.
We also interpret our result as a bridge from a nonconvex optimization problem to a convex one. 
As the second main result of the paper, supposing that the underlying sample space (measurable space) has a measurable diagonal and some additional assumptions on $\psi$, we show that the measurability of a generalized $\psi$-estimator is equivalent to the measurability of
 $\psi$ in its first variable.
We also discuss two particular $\psi$-estimators, the so-called $t$-score moment estimator and the trimmed moment estimator, that are popular in robust statistics.
We can conclude that we derived two new interesting properties of generalized $\psi$-estimators.

\section*{Acknowledgements}
We would like to thank B\'ela Nagy for pointing out the equation \eqref{help_7}, which resulted in a short proof of Lemma \ref{Lem_nfold_diag_measurable}.
We acknowledge the valuable suggestions from the referee.

\section*{Statements and Declarations}

The authors declare that they have no known competing financial interests or personal relationships that could have appeared to influence the work presented in this paper.

\bibliographystyle{plain}

\begin{thebibliography}{10}

\bibitem{BarPal4}
M.~Barczy and {\relax Zs}.~P{\'a}les.
\newblock Basic properties of generalized {$\psi$}-estimators.
\newblock {\em Publ. Math. Debrecen}, 106(3-4):499--524, 2025.

\bibitem{BarPal3}
M.~Barczy and {\relax Zs}.~P{\'a}les.
\newblock Comparison and equality of generalized {$\psi$}-estimators.
\newblock {\em Ann. Inst. Statist. Math.}, 77(2):217--250, 2025.

\bibitem{BarPal2}
M.~Barczy and {\relax Zs}.~P{\'a}les.
\newblock Existence and uniqueness of weighted generalized {$\psi$}-estimators.
\newblock {\em Lith. Math. J.}, 65(1):14--49, 2025.

\bibitem{Bog}
V.~I. Bogachev.
\newblock {\em Measure Theory. {V}ol. {I}, {II}}.
\newblock Springer-Verlag, Berlin, 2007.

\bibitem{BraJonZit}
V.~Brazauskas, B.~L. Jones, and R.~Zitikis.
\newblock Robust fitting of claim severity distributions and the method of
  trimmed moments.
\newblock {\em J. Statist. Plann. Inference}, 139(6):2028--2043, 2009.

\bibitem{DarPal82}
Z.~Dar{\'o}czy and {\relax Zs}.~P{\'a}les.
\newblock On comparison of mean values.
\newblock {\em Publ. Math. Debrecen}, 29(1-2):107--115, 1982.

\bibitem{HewStr}
E.~Hewitt and K.~Stromberg.
\newblock {\em Real and Abstract Analysis}, volume No. 25 of {\em Graduate
  Texts in Mathematics}.
\newblock Springer-Verlag, New York-Heidelberg, third edition, 1975.

\bibitem{Hub64}
P.~J. Huber.
\newblock Robust estimation of a location parameter.
\newblock {\em Ann. Math. Statist.}, 35:73--101, 1964.

\bibitem{Hub67}
P.~J. Huber.
\newblock The behavior of maximum likelihood estimates under nonstandard
  conditions.
\newblock In {\em Proc. {F}ifth {B}erkeley {S}ympos. {M}ath. {S}tatist. and
  {P}robability ({B}erkeley, {C}alif., 1965/66), {V}ol. {I}: {S}tatistics},
  pages 221--233. Univ. California Press, Berkeley, Calif., 1967.

\bibitem{Kle}
A.~Klenke.
\newblock {\em Probability Theory -- A Comprehensive Course}.
\newblock Universitext. Springer, Cham, third edition, 2020.

\bibitem{Kos}
M.~R. Kosorok.
\newblock {\em Introduction to Empirical Processes and Semiparametric
  Inference}.
\newblock Springer Series in Statistics. Springer, New York, 2008.

\bibitem{Pal88a}
{\relax Zs}.~P{\'a}les.
\newblock General inequalities for quasideviation means.
\newblock {\em Aequationes Math.}, 36(1):32--56, 1988.

\bibitem{StePotWalFab}
M.~Stehl\'ik, R.~Potock\'y, H.~Waldl, and Z.~Fabi\'an.
\newblock On the favorable estimation for fitting heavy tailed data.
\newblock {\em Comput. Statist.}, 25(3):485--503, 2010.

\bibitem{Vaa}
A.~W. van~der Vaart.
\newblock {\em Asymptotic {S}tatistics}, volume~3 of {\em Cambridge Series in
  Statistical and Probabilistic Mathematics}.
\newblock Cambridge University Press, Cambridge, 1998.

\end{thebibliography}

\end{document}